\documentclass[article, 11pt, reqno]{amsart}

\usepackage{amssymb, amsmath}
\usepackage[margin=2.5cm]{geometry}
\usepackage{hyperref}
\usepackage[dvipsnames]{xcolor}
\usepackage[color=Purple!50!white, textwidth=20mm]{todonotes}
\newtheorem{theorem}{Theorem}[section]
\newtheorem{lemma}[theorem]{Lemma}

\newtheorem{prop}[theorem]{Proposition}

\newtheorem{cor}[theorem]{Corollary}
\newtheorem{remark}[theorem]{Remark}

\theoremstyle{definition}
\newtheorem{definition}[theorem]{Definition}

\newcommand{\E}{\mathbb{E}}
\newcommand{\F}{\mathbb{F}}

\newcommand{\fknh}{f_k(n,H)}
\newcommand{\fknt}{f_k(n,T)}

\newcommand{\Z}{\mathbb{Z}}
\newcommand{\Prob}{\mathbb{P}}

\numberwithin{equation}{section}

\begin{document}

\title{Repeated patterns in proper colourings}

\author{David Conlon}
\address{Department of Mathematics, California Institute of Technology, Pasadena, CA 91125, USA}
\curraddr{}
\email{dconlon@caltech.edu}
\thanks{Conlon was supported by NSF Award DMS-2054452 and Tyomkyn by ERC Synergy Grant DYNASNET 810115, the H2020-MSCA-RISE Project CoSP- GA No.~823748 and GACR Grant 19-04113}

\author{Mykhaylo Tyomkyn}
\address{Department of Applied Mathematics, Charles University, 11800 Prague, Czech Republic}
\curraddr{}
\email{tyomkyn@kam.mff.cuni.cz}
\thanks{}

\date{}

\maketitle

\begin{abstract}
For a fixed graph $H$, what is the smallest number of colours $C$ such that there is a proper edge-colouring of the complete graph $K_n$ with $C$ colours containing no two vertex-disjoint colour-isomorphic copies, or repeats, of $H$? We study this function and its generalisation to more than two copies using a variety of combinatorial, probabilistic and algebraic techniques. For example, we show that for any tree $T$ there exists a constant $c$ such that any proper edge-colouring of $K_n$ with at most $c n^2$ colours contains two repeats of $T$, while there are colourings with at most $c' n^{3/2}$ colours for some absolute constant $c'$ containing no three repeats of any tree with at least two edges. We also show that for any graph $H$ containing a cycle there exist $k$ and $c$ such that there is a proper edge-colouring of $K_n$ with at most $c n$ colours containing no $k$ repeats of $H$, while, for a tree $T$ with $m$ edges, a colouring with $o(n^{(m+1)/m})$ colours contains $\omega(1)$ repeats of $T$.
\end{abstract}

\section{Introduction} \label{sec:intro}

A considerable body of recent work in extremal combinatorics is devoted to the study of rainbow patterns in proper edge-colourings of complete graphs. To mention two such results (amongst many~\cite{BaPoSu, BePoSu, CoPe, EhGlJo, GaRaWaWo, GlJo, GlKuMoOs, KeYe, KiKuKuOs, MoPoSu, MoPoSu2, Po, PoSu, PoSu2}), there is the work of Alon, Pokrovskiy and Sudakov~\cite{AlPoSu}  showing that any proper edge-colouring of $K_n$ contains a rainbow path of length $n - o(n)$ and the work of Montgomery,  Pokrovskiy and Sudakov~\cite{MoPoSu3} and, independently, Keevash and Staden~\cite{KeSt} resolving a celebrated conjecture of Ringel, one of whose statements involves finding a rainbow copy of any tree with $n$ edges in a particular proper edge-colouring of $K_{2n+1}$. For the most part, this recent work has focused on finding large structures in proper edge-colourings. We instead study small structures, our aim being to understand when a proper edge-colouring contains two or more repeats of a particular graph $H$.

To be more precise, we say that two copies of a graph $H$ in a colouring of $K_n$ are \emph{colour isomorphic} if there exists an isomorphism between them preserving the colours. The following function is our main object of study.

\begin{definition}
For $k,n\geq 2$ and a graph $H$, define $\fknh$ to be the smallest integer $C$ such that there is a proper edge-colouring of $K_n$ with $C$ colours containing no $k$ vertex-disjoint colour-isomorphic copies (or `repeats') of $H$.
\end{definition}  

We make several remarks about this definition. First, one could, in principle, ask the same question without the restriction to proper colourings. However, this changes the character of the question completely. Indeed, consider the colouring of the complete graph on vertex set $\{1, 2, \dots, n\}$ where we colour the edge $ij$ with $i < j$ by the colour $i$. Then this is a colouring with $n-1$ colours which does not even contain two disjoint edges of the same colour. On the other hand, when we restrict to proper colourings, we have that $f_k(n, K_2) \geq \lceil \frac{1}{k-1}\binom{n}{2}\rceil$ by a straightforward application of the pigeonhole principle. For $n$ sufficiently large in terms of $k$, it also follows from several well-known decomposition results, such as Gustavsson's theorem~\cite{Gu}, that this bound is tight.

Our second remark collects together several simple observations that we will use throughout. 

\begin{remark}\label{rem:trivial}
	The quantity $\fknh$ is monotone increasing in $n$, but decreasing in $k$ and in $H$ (with respect to taking subgraphs). Moreover, since every proper colouring has at least $n-1$ colours,
	$$\binom{n}{2}\geq \fknh\geq n-1.$$ 
\end{remark}

Finally, although our definition contains no requirement that the copies of $H$ should be rainbow, all of the results where we find repeats of a particular graph $H$ remain true up to a constant factor if we insist that each copy is rainbow. This then brings our work more fully in line with the body of research discussed at the outset. With these preliminaries out of the way, we now describe our main results.

In the classical Tur\'an problem, the growth rate of the extremal function ex$(n,H)$ is subject to a well-known trichotomy. Namely, non-bipartite graphs, bipartite graphs with a cycle and forests satisfy $\textrm{ex}(n,H)=\Theta(n^2)$, $n^{1+\Omega(1)} \leq \textrm{ex}(n,H)\leq n^{2-\Omega(1)}$ and $\textrm{ex}(n,H)=\Theta(n)$, respectively. Our first theorem shows that something broadly similar holds for $f_2(n, H)$, although, unlike the extremal function, our function can, and usually does, degenerate for bipartite graphs with a cycle. Note that here and throughout, all terms in the $O$-notation are to be interpreted with respect to $n$, with all other variables treated as constants.

\begin{theorem}\label{thm:f2}
	The growth rate of $f_2(n,H)$ satisfies:
	\begin{itemize}
\item[(i)] $f_2(n,H)=\Theta(n^2)$ if $H$ is a forest. Otherwise, $f_2(n,H)=O(n^{2-\Omega(1)})$. 
\item[(ii)] If $H$ is non-bipartite, then $f_2(n,H)\leq n+1$. 
\item[(iii)] If $H$ is bipartite and $e(H)\geq 2|H|-2$, then $f_2(n,H)=\Theta(n)$.
\item[(iv)] There exist bipartite graphs $H$ with $n^{1+\Omega(1)}\leq f_2(n,H)\leq n^{2-\Omega(1)}$.
\end{itemize}
\end{theorem}  

For three or more repeats, the class of graphs for which we know that $f_k(n, H) = O(n)$ grows as $k$ increases. In fact, for any graph $H$ containing a cycle, we can show, by using a variant of Bukh's random algebraic method~\cite{Bu}, that there exists $k$ such that $\fknh = O(n)$.

\begin{theorem} \label{thm:cyclim}
For any graph $H$ containing a cycle, there exists $k$ such that $\fknh = O(n)$.
\end{theorem}

For trees, the situation is much more involved, as spelled out in the next theorem, whose proof relies on a mixture of novel combinatorial and algebraic methods.

\begin{theorem}\label{thm:trees}
For any tree $T$ with $m$ edges and any $k \geq 3$:
\begin{itemize}
\item[(i)] $f_k(n, T) = \Omega(n^{k/(k-1)})$. Moreover, if $T$ has at least two edges, then $f_3(n,T)= \Theta(n^{3/2})$.
\item[(ii)] $f_k(n, T) = \Omega(n^{(m+1)/m})$ and there exists $k'$ such that $f_{k'}(n, T) = O(n^{(m+1)/m})$.
\end{itemize}
\end{theorem}

The paper is structured as follows. In Section~\ref{sec:prelim}, we collect some general upper and lower bounds for $\fknh$. We then proceed to study trees in Section~\ref{sec:trees}. In Section~\ref{sec:bipartite}, we study even cycles. We conclude, in Section~\ref{sec:outlook}, by discussing a broad range of open problems suggested by our work. Unless otherwise specified, the term `colouring' will refer to proper edge-colourings of a complete graph. 

\section{Preliminary Observations}\label{sec:prelim}

Since it is not entirely obvious, we first verify that finding $\fknh$ is indeed an extremal problem.
\begin{lemma}
	For any $k,n,H$ and $\binom{n}{2}\geq C> \fknh$, there exists a $C$-colouring of $K_n$ which does not contain $k$ repeats of $H$.
\end{lemma}
\begin{proof}
Suppose for contradiction that the statement does not hold for some $C>\fknh$ and consider an $\fknh$-colouring of $K_n$ which avoids $k$ repeats of $H$. In this colouring, take any $C-\fknh$ edges from colour classes of size at least $2$, leaving at least one edge in each colour class and recolour each of these edges with an entirely new colour. By assumption, the resulting $C$-colouring contains $k$ repeats of $H$. Since each new colour class has only one edge, the repeats must be in the old colours, a contradiction. 
\end{proof}

The following lemma effectively reduces the problem to connected graphs. 

\begin{lemma}\label{lem:components}
Suppose that $H$ is a graph with connected components $H_1,\dots, H_\ell$. Then, for all $k\geq 2$,
\begin{equation*}\label{eq:components}
\fknh=\Theta(\min_{i\in [\ell]}\{f_k(n,H_i)\}).
\end{equation*}
\end{lemma}
	
\begin{proof}
Since containing no $k$-fold repeat of one of the $H_i$ implies that there is no $k$-fold repeat of $H$, we immediately have 
\begin{equation}\label{eq:trivialmin}
\fknh\leq \min_{i\in [\ell]}\{f_k(n,H_i)\}.
\end{equation}
For the other direction, let $k\geq 2$ be fixed. We will prove the result by induction on $\ell$, noting first that the statement is obviously true for $\ell=1$. Given $\ell\geq 2$, $H$ and $n$, we may assume, without loss of generality, that 
$$f_k(n,H_\ell)=\min_{i\in [\ell]}\{f_k(n,H_i)\}$$
and let $H':=H\setminus H_\ell$. By the induction hypothesis, there exists a constant $c=c(k,\ell, H_1,\dots,H_{\ell-1})$ such that 
$$f_k(n, H')\geq c\cdot \min_{i\in[\ell-1]}f_k(n,H_i)\geq c \cdot f_k(n,H_\ell).
$$
Observe now that  
\begin{equation}\label{eq:mincomp}
f_k(n,H)\geq \min\{f_k(n,H_\ell),f_k(n-k|H_\ell|,H')\},
\end{equation}
as otherwise we can find a $k$-repeat of $H_\ell$ and then a $k$-repeat of $H'$ on the remaining vertices, resulting in a $k$-repeat of $H$. If $f_k(n,H_\ell)\leq f_k(n-k|H_\ell|,H')$, then, by~\eqref{eq:mincomp} and ~\eqref{eq:trivialmin}, $f_k(n,H)=f_k(n,H_\ell)$ and we are done. So suppose now that $f_k(n,H_\ell)> f_k(n-k|H_\ell|,H')$, in which case~\eqref{eq:mincomp} becomes
\begin{equation}\label{eq:amostfull}
f_k(n,H)\geq f_k(n-k|H_\ell|,H').
\end{equation}
The right-hand side of~\eqref{eq:amostfull} is at least $f_k(n,H')-k|H_\ell|n$, as can be seen by extending a colouring of the complete graph on $n-k|H_\ell|$ vertices containing no $k$-repeat of $H'$ to a colouring of $K_n$ with the same property by using a new unique colour for each edge (for a total of at most $k|H_\ell|n$ new colours). 
Therefore,
\begin{equation}\label{eq:split_fk}
f_k(n,H)\geq f_k(n-k|H_\ell|,H')\geq f_k(n,H')-k|H_\ell|n\geq c\cdot f_k(n,H_\ell)-k|H_\ell|n.
\end{equation}
If $2(k/c)|H_\ell|n<f_k(n,H_\ell)$, then, by~\eqref{eq:split_fk},
$$f_k(n,H)\geq c f_k(n,H_\ell)-\frac{c}{2}f_k(n,H_\ell)=\frac{c}{2}f_k(n,H_\ell).
$$ 
On the other hand, by Remark~\ref{rem:trivial}, we have $f_k(n,H)\geq n-1\geq n/2$. Hence, if $ 2(k/c)|H_\ell|n\geq f_k(n,H_\ell)$, then
$$f_k(n,H)\geq \frac{n}{2}=\frac{2(k/c)|H_\ell|n}{4(k/c)|H_\ell|}\geq \frac{c}{4k|H_\ell|}f_k(n,H_\ell),
$$
completing the induction step and the proof.
\end{proof}

Next, we quickly resolve the case of non-bipartite graphs, proving Theorem~\ref{thm:f2}(ii). 

\begin{theorem}\label{thm:oddcycle}
	If $H$ contains an odd cycle, then $\fknh \leq n+1$.
\end{theorem}

\begin{proof}
It suffices to show this for $k=2$ and $H = C_{2 \ell+1}$. To this end, consider the following well-known colouring $c$. Suppose that $n=2p+1$ is an odd number, let $V(G)=\Z_n$ and put $c(a,b):=a+b \mod n$. Clearly this is a proper colouring using $n$ colours. Moreover, whenever there is a copy of $C_{2\ell+1}$ with vertices $v_0,\dots,v_{2\ell+1}=v_0$, labelled along the cycle, we have $c_i:=c(v_i,v_{i+1})=v_i+v_{i+1} \mod n$ for each $i$. The vertices $v_0, \dots , v_{2\ell}$ are then uniquely determined by the colours $c_0,\dots,c_{2\ell}$, since $v_0=(c_0-c_1+c_2-\dots +c_{2\ell})/2 \mod n$ and a similar identity holds for the other $v_i$. Hence, no two distinct, let alone vertex-disjoint, cycles will be colour isomorphic. This proves that $\fknh\leq n$ when $n$ is odd. For even $n$, monotonicity implies that $\fknh\leq f_k(n+1,H)\leq n+1$. 
\end{proof}	

For odd $n$, the proof of Theorem~\ref{thm:oddcycle} gives $f_k(n,H)\leq n$, which is best possible, as this is the minimum number of colours in any proper colouring when $n$ is odd. For even $n$, the bounds in Remark~\ref{rem:trivial} and Theorem~\ref{thm:oddcycle} differ by $2$. We refer the reader to Section~\ref{sec:outlook} for further discussion. 
	
For upper-bound constructions, the following observation will be very useful. A not necessarily proper edge-colouring of $K_n$ is called \emph{$b$-bounded} if each colour class is a graph of maximum degree at most $b$. 

\begin{prop}\label{prop:b-bounded}
For any graph $H$ and integers $b,k,n\geq 2$, any $b$-bounded $C$-colouring of $K_n$ without $k$ repeats of any (not necessarily proper) colouring of $H$ can be refined to a proper colouring with at most $(b+1)C$ colours and no $k$-repeat of $H$.  	
\end{prop} 

\begin{proof}
Since each colour class is a graph of maximum degree $b$, Vizing's theorem implies that the edges in each colour class can be recoloured with at most $b + 1$ new colours so that the new colouring is proper. A $k$-repeat of $H$ in the new colours would imply the same in the old colours and is therefore impossible.
\end{proof}	

The next theorem provides, via a simple probabilistic argument, a general upper bound for $\fknh$. In light of Theorem~\ref{thm:oddcycle}, it is only meaningful for bipartite $H$. 

\begin{theorem}\label{thm:LLL}
	For any graph $H$ with $v$ vertices and $e$ edges,
	$$\fknh=O(\max\{n,n^{\frac{kv-2}{(k-1)e}}\}).$$
\end{theorem}

\begin{proof}
	This is a standard application of the Lov\'asz Local Lemma (see, for instance,~\cite{AlSp}). 
	We would like to find a probability $p=p(n)$ and assign each edge of $K_n$ independently   one of $1/p$ colours so that with positive probability the following will hold:
	\begin{itemize}
		\item[(P1)] There is no $k$-fold repeat of any (not necessarily proper) colouring of $H$. 
		\item[(P2)] The colouring is $(kv-2)$-bounded.
	\end{itemize}
	For this, we define two collections of `bad' events, as follows. 
	Given $k$ disjoint copies of $H$ in $K_n$, say $H_1,\dots,H_k$, let $A(H_1,\dots, H_k)$ be the event that $H_1,\dots,H_k$ are colour isomorphic. We have $$q_A:=\Prob(A(H_1,\dots,H_k))=\Theta(p^{(k-1)e}).$$
	For every copy $S$ of $S_\ell$, the star with $\ell := kv-1$ edges, we define $B(S)$ to be the event that all edges of $S$ receive the same colour. We have $$q_B:=\Prob(B(S))=p^{\ell-1}.$$  
	
	For every fixed $H_1,\dots,H_k$, the event $A(H_1,\dots,H_k)$ is independent of any event where the subgraphs involved do not share an edge with any of $H_1,\dots, H_k$. Therefore, the number of dependent events is at most
	$$\Delta = O(n^{kv-2} +n^{(\ell+1)-2}) = O(n^{\ell-1}).$$ 
	Similarly, each $B(S)$ is independent from all but at most $\Delta = O(n^{\ell-1})$ other events. 
	By the Local Lemma, in order to establish that $\Prob(\bigcap \overline{A}\cap \bigcap \overline{B})>0$, it suffices to show that $\max\{q_A,q_B\} < 1/e\Delta$. This will be the case when 
	$$p \leq \gamma \min\{n^{-1}, n^{-\frac{\ell-1}{(k-1)e}}\}$$
	for some constant $\gamma=\gamma(k,v,e)>0$. Therefore, there exists a $(kv-2)$-bounded colouring with $O(\max\{n,n^{\frac{kv-2}{(k-1)e}}\})$ colours satisfying (P1). By Proposition~\ref{prop:b-bounded}, it can be refined to a proper $O(\max\{n,n^{\frac{kv-2}{(k-1)e}}\})$-colouring with no $k$-repeat of $H$, as desired.
\end{proof}

We immediately deduce the following corollary, which includes Theorem~\ref{thm:f2}(iii) as a special case. 

\begin{cor}\label{cor:LLLbipartite}
If $H$ is a bipartite graph with $e(H)\geq \frac{k}{k-1}|H|-\frac{2}{k-1}$, then 
$$f_k(n,H)=\Theta(n).$$
\end{cor}

For $e(H) \geq |H| + 1$, this already implies that there exists $k = k(H)$ such that $f_k(n, H) = O(n)$. We will later show that this holds for all graphs containing a cycle, but using Theorem~\ref{thm:LLL} gives an explicit upper bound for the smallest $k$ such that $\fknh = O(n)$, namely, $k \leq \lceil\frac{e-2}{e-v}\rceil$, that our later approach does not provide.

\section{Trees}\label{sec:trees}

Applied to trees, Theorem~\ref{thm:LLL} reads as follows.

\begin{cor}\label{cor:LLLtree}
	For any tree $T$ with $m$ edges, $$\fknt=O(n^{\frac{k(m+1)-2}{(k-1)m}}).$$	
\end{cor}

On the other hand, we have the following general lower bound for trees, showing that Corollary~\ref{cor:LLLtree} is not far from the truth for large trees.

\begin{theorem}\label{thm:paths}
For any tree $T$ with $m$ edges, any $k\geq 2$ and $n$ sufficiently large, 
$$\fknt\geq n^{\frac{k}{k-1}}/q,$$ 
where $q = (4k(km+1)(k^2-k+1))^{\frac{1}{k-1}}$. 
\end{theorem}
	
\begin{proof}
Fix $k$ and $T$, let $q$ be as above and suppose that for some large $n$ we have a colouring of $G=K_n$ with $C=n^{\frac{k}{k-1}}/q$ colours.
Let $F$ be the following auxiliary graph. Set $V(F)=\binom{V(G)}{k}$, that is, each vertex of $F$ is a  vertex subset of $G$ of order $k$. Two vertices of $F$, corresponding to sets $U=\{u_1,\dots,u_k\}$ and $W=\{w_1,\dots,w_k\}$, are connected by an edge if $U\cap W=\emptyset$ and the bipartite graph $G[U,W]$ has a matching of size $k$ which is monochromatic. 

By the above definition, each monochromatic $k$-matching gives rise to $2^{k-1}$ edges of $F$, while each edge of $F$ is accounted for by at most $k$ different $k$-matchings. This implies that
$$e(F)\geq \frac{2^{k-1}}{k}E_k,
$$
where $E_k$ is the number of monochromatic $k$-matchings in $G$.

Note now that for the average degree of $F$, we have
$$d_{avg}(F)=\frac{2e(F)}{|F|}\geq\frac{2^kE_k}{k\binom{n}{k}}.$$
By the convexity of binomial coefficients, 
$$E_k\geq {C\binom{\frac{1}{C}\binom{n}{2}}{k}}.
$$
Thus,
$$
d_{avg}(F)\geq \frac{2^kC\binom{\frac{1}{C}\binom{n}{2}}{k}}{k\binom{n}{k}}=(1+o(1))\frac{n^{k}}{kC^{k-1}}=(1+o(1))\frac{q^{k-1}}{k}. 
$$ 
By a folklore fact, $F$ therefore contains a subgraph $F'$ with
$$d:=\delta(F')\geq \frac{d_{avg}(F)}{2}\geq (1+o(1))\frac{q^{k-1}}{2k}.$$

Let $v_0,\dots, v_m$ be the vertices of $T$ ordered so that each $v_i$ is a leaf attached to $T[v_0,\dots,v_{i-1}]$. We claim that there exists an embedding $\phi:T\rightarrow F'$ such that the $k$-vertex sets of $G$ corresponding to $\phi(v_0),\dots, \phi(v_m)$ are disjoint. 
 
We let $\phi(v_0)$ be an arbitrary vertex of $F'$ and proceed inductively. Suppose that for some $1\leq i\leq m$ we have defined $\phi(v_0),\dots,\phi(v_{i-1})$ and let $v_j$ be the unique neighbour of $v_i$ in $T[v_0,\dots, v_{i-1}]$. Take any $d$ vertices in $N_{F'}(\phi(v_j))$ and consider the underlying $k$-sets in $V(G)$. This defines a $k$-uniform hypergraph $H$. 

Observe that the maximum degree in $H$ is at most $k$, since for every vertex $u\notin \phi(v_j)$ there are at most $k$ monochromatic matchings in $G$ connecting $\phi(v_j)$ and a set containing $u$. Thus, the line graph of $H$ has maximum degree at most $k(k-1)$, so we can properly vertex colour it greedily with at most $k^2-k+1$ colours. Take the largest colour class: this is a matching $\mu \subseteq E(H)$ of size at least $d/(k^2-k+1)$. But
$$\frac{d}{k^2-k+1}\geq (1+o(1))\frac{q^{k-1}}{2k(k^2-k+1)} = (2+o(1))(km+1).$$
Therefore, when $n$ is sufficiently large, we have $|\mu| \geq km+1\geq ki+1$, so there exists $e\in \mu$ such that the underlying $k$-vertex set in $G$ satisfies $e \cap \phi(v_\ell)=\emptyset$ for all $\ell=0,\dots, i-1$. We can therefore let $\phi(v_i):=e$, completing the induction step. The resulting copy of $\phi(T)\subseteq F'$ gives $k$ vertex-disjoint colour-isomorphic copies of $T$ in the original colouring of $K_n$, completing the proof.
\end{proof}	

In particular, when $k = 2$, we have the following result.

\begin{cor} \label{cor:trees}
For any tree $T$ with $m$ edges and $n$ sufficiently large,
$$f_2(n,T)\geq \frac{1}{24(2m+1)}n^2.$$
\end{cor}	

This also completes the proof of Theorem~\ref{thm:f2}(i). Indeed, if $H$ is a forest, then $f_2(n,H)=\Theta(n^2)$ by Corollary~\ref{cor:trees} and Lemma~\ref{lem:components}. Conversely, if $H$ is not a forest, it has a connected component $H'$ with $e(H')\geq |H'|$. By Theorem~\ref{thm:LLL}, 
$f_2(n,H')=O(n^{2-\Omega(1)})$, so the same then holds for $H$ by Lemma~\ref{lem:components}. 

We now observe that the bound in Corollary~\ref{cor:trees} is tight up to an absolute constant.

\begin{lemma}
For every tree $T$ with $m\geq 2$ edges, 
$$f_2(n,T)\leq (1+o(1))\frac{n^2}{2m}.$$ 
\end{lemma}

\begin{proof}
By Wilson's theorem~\cite{Wi}, there is an edge decomposition of all but $o(n^2)$ edges of $K_n$ into at most $\binom{n}{2}/\binom{2m+1}{2}$ copies of $K_{2m+1}$. We colour each of the $o(n^2)$ edges in a unique colour and decompose each of the copies of $K_{2m+1}$ into at most $2m+1$ (near-)perfect matchings, each of which we view as a separate colour class. Since any two copies of $K_{2m+1}$ share at most one vertex, there will be no repeat of $T$ with colours from two different $K_{2m+1}$. On the other hand, since $2m+1 < 2(m+1)$ and $T$ has $m+1$ vertices, there cannot be two vertex-disjoint copies of $T$ inside any $K_{2m+1}$. 
\end{proof}	

Next, we show that the bound in Theorem~\ref{thm:paths} is essentially tight when $k = 3$, thus completing the proof of Theorem~\ref{thm:trees}(i).

\begin{theorem}\label{thm:3rep}
For any tree $T$ with at least $2$ edges, $f_3(n,T)= O(n^{3/2})$.
\end{theorem}

\begin{proof}
Suppose that $A$ is a vertex set of order $n \leq q^2 = (1 + o(1))n$ for some prime $q$, noting that such a prime exists by the prime number theorem. Our aim is to find a proper colouring of the complete graph on $A$ with $O(n^{3/2})$ colours such that no three vertices are incident to the same two colours. In particular, this will imply that the colouring has no $3$-fold repeat of the star with two edges and, hence, no $3$-fold repeat of any tree $T$ with more than one edge. 
  
To achieve this, let $A$ be indexed by $\F_q^2$, where $\F_q$ denotes the finite field of order $q$. We first deal with certain `degenerate' edges in $\binom{A}{2}$. For us, these will be edges between two vertices $u=(a,b)$ and $v=(c,d)$, where either $a=c$, $a=1$ or $c=1$. Note that there are $O(q^3)=O(n^{3/2})$ such edges, so we may colour each of them with a unique colour. To each remaining edge $\{u,v\}\in \binom{A}{2}$ with $u=(a,b)$ and $v=(c,d)$, we assign the colour $(x_1,x_2,x_3) \in \F_q^3$ (so that there are at most $q^3 = O(n^{3/2})$ additional colours) satisfying the equations $x_1+x_2a+x_3a^2=b$, $x_1+x_2c+x_3c^2=d$ and $x_1+x_2+x_3=a+c$. This system of three equations has a unique solution $(x_1,x_2,x_3)$, since the underlying matrix is a Vandermonde matrix (and, since $a, c$ and $1$ are all distinct, the matrix is non-singular). Now, for a fixed $u=(a,b)$, how many non-degenerate edges $uv$ with $v = (c, d)$ receive the colour $(x_1,x_2,x_3)$? Each such edge has to satisfy $c=x_1+x_2+x_3-a$, so $c$ is fixed by the choice of $u$ and the colour $(x_1, x_2, x_3)$. Moreover, since $d = x_1+x_2c+x_3c^2$, $d$ is also fixed. Hence, there is at most one edge of each colour adjacent to $u$. That is, we have a proper colouring. 

Note now that each non-degenerate colour $(x_1,x_2,x_3)$ incident to a vertex $(a,b)$ satisfies the equation $x_1+x_2a+x_3a^2=b$. Therefore, since there is at most one quadratic passing through any three points in $\F_q^2$, no three vertices can be incident to the same two colours.
\end{proof}

\begin{remark}
If $H$ is not a forest, $f_3(n, H)$ is always polynomially smaller than $n^{3/2}$.
Indeed, if $H$ is not a forest, it has a connected component $H'$ with $e(H')\geq |H'|$. By Theorem~\ref{thm:LLL}, 
$f_3(n,H')=O(n^{3/2-\Omega(1)})$, so the same then holds for $H$ by Lemma~\ref{lem:components}.
\end{remark}

We conclude this section by establishing the first part of Theorem~\ref{thm:trees}(ii).

\begin{prop}\label{prop:infinite}
For any tree $T$ with $m$ edges and any $k\geq 2$, $\fknt=\Omega(n^{\frac{m+1}{m}})$.	
\end{prop} 

\begin{proof}
For convenience of notation, we will instead prove the equivalent statement that a proper $C$-colouring of $K_n$ with $C=o(n^{\frac{m+1}{m}})$ colours contains $\omega(1)$ repeats of $T$. The number of copies of $T$ in $K_n$ is $\Theta(n^{m+1})=\omega(C^m)$. Thus, some $m$-tuple of colours will be repeated  $\omega(1)$ times. Since the colouring is proper, every vertex appears in at most $|T|=m+1$ copies of $T$ with a fixed colouring. Hence, $\omega(1)$ of the above repeats will be vertex disjoint. 
\end{proof}	

\section{Cycles}\label{sec:bipartite}

We begin this section by proving two of our main results at once, Theorem~\ref{thm:cyclim} and the second part of Theorem~\ref{thm:trees}(ii).

\begin{theorem}\label{thm:evencycle}
	For every bipartite graph $H$ containing a cycle, there exists $k=k(H)$ such that 
	$$\fknh=O(n).$$ 
	For every $m$-edge tree $T$, there exists $k=k(T)$ such that $$\fknt=O(n^{\frac{m+1}{m}}).$$  
\end{theorem}

In the proof of this theorem, we will make use of Bukh's random algebraic method~\cite{Bu} (see also~\cite{BuCo, Co}). It will be useful to briefly recall some basic terminology from algebraic geometry. Let $\overline{\mathbb{F}}_q$ stand for the algebraic closure of $\mathbb{F}_q$. A \emph{variety} over $\overline{\mathbb{F}}_q$ is a set of the form 
$$W=\{x\in \overline{\mathbb{F}}_q^t:f_1(x)=\dots=f_s(x)=0\}$$
for a collection of multivariate polynomials $f_1,\dots, f_s: \overline{\mathbb{F}}_q^t\rightarrow\overline{\mathbb{F}}_q$. We also write $W(\mathbb{F}_q):=W\cap \mathbb{F}_q^t$. The variety $W$ is said to be \emph{defined} over $\mathbb{F}_q$ if the coefficients of $f_1, \dots, f_s$ are in $\mathbb{F}_q$. Finally, we say that $W$ has \emph{complexity} at most $M$ if $s, t$ and the degrees of the polynomials $f_i$ are all at most $M$. We will repeatedly use the following lemma, a consequence of the well-known Lang--Weil bound~\cite{LaWe}.  

\begin{lemma}[\cite{BuCo}, Lemma 2.7]\label{lem:27}
Suppose $W$ and $D$ are varieties over $\overline{\F}_q$ of complexity at most $M$ which are defined over $\F_q$. Then one of the following holds for all $q$ sufficiently large in terms of $M$:
\begin{itemize}
	\item $|W(\F_q)\setminus D(\F_q)| > q/2$ or
	\item $|W(\F_q)\setminus D(\F_q)| < c$, where $c=c_M$ depends only on $M$.
\end{itemize}
\end{lemma}

\begin{proof}[Proof of Theorem~\ref{thm:evencycle}]
	By Remark~\ref{rem:trivial}, it suffices to prove that, for a cycle or a tree $H$ with $(v,e):=(|V(H)|,|E(H)|)$, there exists $k$ such that
	$$\fknh=O(n^{{v}/{e}}).$$ 
	
	Suppose first that $H$ is a cycle, so $v=e$. Let $n \leq q=(1+o(1))n$ be a prime, guaranteed for large $n$ by the prime number theorem. We also let $d$ and $t$ be positive integers, with $d$ chosen to be sufficiently large in terms of $t$ and $t$ sufficiently large in terms of $v$.
	We write $\mathcal{P}_d$ for the set of all two-variable polynomials over $\F_q$ of degree at most $d$ and let $g$ be a polynomial taken from $\mathcal{P}_d$ uniformly at random. Such a polynomial can be generated by selecting a coefficient $a_{d_1d_2}\in \F_q$ independently at random for each pair of non-negative integers $(d_1,d_2)$ with $d_1+d_2\leq d$ and writing 
	$$g(X_1,X_2):=\sum a_{d_1d_2} X_1^{d_1}X_2^{d_2}.$$
We will make repeated use of the following property of such random polynomials. Because we will need it below, we state the result for the more general case of $t$-variate random polynomials.

\begin{lemma}[\cite{BuCo}, Lemma 2.3]\label{lem:23}
Suppose that $q>\binom{m}{2}$ and $d\geq m-1$. Then, if $g$ is a random polynomial from $\mathcal{P}_d$ and $x_1,\dots,x_m$ are $m$ distinct points in $\F_q^t$,
$$\Prob[g(x_i=0) \text{ for all } \, i=1,\dots,m]=1/q^m.
$$	
\end{lemma}

Define a $q$-colouring of $G=K_{[n]}$ by 
assigning each edge $ij$ with $i<j$ the colour $g(i,j)$. 
We claim that with positive probability this colouring has the following properties:
\begin{itemize}
	\item[(Q1)] Every (not necessarily proper) colouring of $H$ occurs fewer than $k=k(H)$ times.
	\item[(Q2)] The colouring is $2d$-bounded.
\end{itemize}	
Given a colouring satisfying these conditions, we can apply Proposition~\ref{prop:b-bounded} to obtain a proper colouring with at most $(2d+1)q=O(n)$ colours and no $k$-repeats of $H$.

Suppose now that $H$ has vertices $1, \dots, v$. We consider all copies of $H$ in $G$ of fixed \emph{order and colour type}. That is, given an ordering $(\tau_1, \dots, \tau_v)$ of $1, \dots, v$ and a sequence of $e$ colours $(b_{ij}:ij\in E(H))$, we consider only those copies of $H$ in $G$ where the images $y_1,\dots, y_{v}$ of $1, \dots, v$ in $[n]$ satisfy $y_{\tau_1} < \dots < y_{\tau_v}$ and the edge between $y_i$ and $y_j$ is coloured with colour $b_{ij}$ for all $ij \in E(H)$. 

Let $W$ be the variety defined by the system of equations
\begin{equation}\label{eq:variety}
g(x_{i,j,1},x_{i,j,2})=b_{ij} \mbox{ for } ij\in E(H)
\end{equation}
with $v$ variables $x_1,\dots,x_v$, where each $(x_{i,j,1},x_{i,j,2})$ equals $(x_{i},x_{j})$ if $\tau_i<\tau_j$ and $(x_{j},x_{i})$ otherwise. Let $D$ be the variety defined by the equation 
$$\prod_{1\leq i\neq j \leq v}(x_i-x_j)=0.
$$
Note that the number of copies of $H$ of given type is at most $S=|W(\F_q)\setminus D(\F_q)|$, i.e., the random variable counting the number of `non-degenerate' solutions to the system of equations~\eqref{eq:variety}.

By Lemma~\ref{lem:23}, for any $s \leq d + 1$ distinct points $y_1,\dots,y_s$ in $\F_q^2$ and any $b_1,\dots,b_s\in \F_q$, we have
$$\Prob[g(y_i) = b_i \mbox{ for all } i = 1,\dots,s] = 1/q^s.$$
Now consider $S^t$, observing that it counts ordered collections of $t$ (potentially overlapping or identical) copies of $H$. Note that the graph $H_t$ spanned by any collection of $t$ copies of $H$ will satisfy $e(H_t)\geq |H_t|$, since each connected component of $H_t$ contains a cycle.
Therefore, for $d$ sufficiently large in terms of $v$ and $t$, 
$$\E[S^{t}]\leq \sum_{s=v}^{vt}\frac{s^{vt} q^s}{q^{s}}< (vt)^{2vt},$$
where the factor $s^{vt}$ is an upper bound for the number of ordered ways of placing $t$ ordered copies of $H$ on a fixed set of $s$ vertices.

Since $W$ and $D$ are defined over $\mathbb{F}_q$ and their complexities are bounded by a function of $d$ and $v$, Lemma~\ref{lem:27} implies  that there exists a constant $K=K(d, v)$ such that the random variable $S=|W(\F_q)\setminus D(\F_q)|$ satisfies either $S < K$ or $S > q/2$. Thus, by Markov's inequality,
$$\Prob[S \geq K]=\Prob[S>q/2]=\Prob[S^t>(q/2)^t]\leq \frac{\E[S^t]}{(q/2)^t}< (vt)^{2vt}(q/2)^{-t}.
$$  
This gives an upper bound on the probability of having at least $K$ distinct copies of $H$ for any given order and colour type. Since there are at most $v!$ different order types and at most $q^e=q^v$ choices for the sequence of colours, the union bound implies that, for $t$ sufficiently large in terms of $v$, there are with high probability fewer than $K$ distinct copies of $H$ for each order and colour type. Therefore, writing $k = v! K$, we see that with high probability the colouring has fewer than $k$ copies of $H$ for each colour type and so it satisfies (Q1) with probability $1-o(1)$.

To see that it also satisfies (Q2), we have to show that with high probability, for each fixed $i$ and $a$, the equations $g(x,i)=a$ and $g(i,y)=a$ have at most $d$ solutions, i.e., that each of $g(x,i)$ and $g(i,y)$ has a non-constant coefficient not equal to zero. 

Consider $g(x,i)$ and note that, when viewing $g$ as a polynomial in $\F_q[X_2][X_1]$, the coefficient of $X_1^p$ is a random univariate polynomial $h_p(X_2)$ of degree $d-p$ with the coefficient for each different $p$ being independent. Hence, the probability that $h_p(i)$ is $0$ for all $p \geq 1$ is $q^{-d}$. Since there are at most $q^2$ choices for $i$ and $a$, the union bound implies that with high probability $g(x,i)=a$ has at most $d$ solutions for all $i$ and $a$. Together with the analogous result for $g(i, y) = a$, this establishes (Q2) with probability $1-o(1)$. By the union bound, with positive probability the colouring satisfies both properties (Q1) and (Q2), as claimed. 

Suppose now that $H$ is a tree with $m\geq 1$ edges. Let $q$ be a prime with $n^{1/m} \leq q=(1+o(1))n^{1/m}$. Let $d$ and $t$ again be positive integers, with $d$ chosen to be sufficiently large in terms of $t$ and $t$ sufficiently large in terms of $m$, while $\mathcal{P}_d$ is now the set of all $2m$-variable polynomials over $\F_q$ of degree at most $d$. Let $g_1,\dots, g_{m+1}$ be $m+1$ polynomials taken from $\mathcal{P}_d$ independently and uniformly at random.

Fix an ordering $\prec$ on $\F_q^m$ and associate $V(K_n)$  with a subset of $\F_q^m$. We define a $q^{m+1}$-colouring on $\binom{\F_q^m}{2}$ by assigning each edge $ij$ with $i\prec j$ the colour $(g_1(i,j),\dots,g_{m+1}(i,j))$, where we interpret $(i,j)$ as an element of $\F_q^m\times \F_q^m \cong \F_q^{2m}$.  As in the cycle case, we claim that with positive probability the following statements hold (in which case we are again done by Proposition~\ref{prop:b-bounded}):
\begin{itemize}
	\item[(Q1)] Every (not necessarily proper) colouring of $H$ occurs fewer than $k=k(H)$ times.
	\item[(Q$2'$)] The colouring is $k'$-bounded for some $k'=k'(H)$.
\end{itemize}	

Define order and colour types as before, but with respect to the ordering $\prec$. For a fixed order type $(\tau_1,\dots,\tau_{m+1})$ and colour type $(b_{ij} : ij\in E(H))$, let $W$ be the variety defined by the system of equations 
\begin{equation*} 
	g_\ell(x_{i,j,1},x_{i,j,2})=b_{ij}^{(\ell)} \mbox{ for } ij\in E(H) \mbox{ and } \ell\in[m+1]
\end{equation*}
with $m+1$ variables $x_1,\dots,x_{m+1} \in \F_q^m$, where each $(x_{i,j,1},x_{i,j,2})$ equals $(x_{i},x_{j})$ if $\tau_i<\tau_j$ and $(x_{j},x_{i})$ otherwise and $b_{ij}=(b_{ij}^{(1)}, \dots , b_{ij}^{(m+1)})$. For each $\{i,j\}\in \binom{m+1}{2}$, let $D_{ij}$ be the variety defined by the system of equations 
$$x_i^{(\ell)} = x_j^{(\ell)} \mbox{ for } \ell\in[m],
$$
where $x_i^{(\ell)}$ and $x_j^{(\ell)}$ are the $\ell$-th coordinates of $x_i$ and $x_j\in \F_q^m$, respectively, and let 
$$D=\bigcup_{1\leq i\neq j\leq m + 1}D_{ij}.
$$ 
We again have that the number of copies of $H$ of given type is bounded above by $S=|W(\F_q)\setminus D(\F_q)|$.

By Lemma~\ref{lem:23} and the independence of the $g_\ell$, for any $s \leq d+1$ distinct points $y_1,\dots,y_s$ in $\F_q^{2m}$ and any choice of $b_i^{(\ell)} \in \F_q$ for each $i\in [s]$ and $\ell\in [m+1]$, we have
$$\Prob[g_\ell(y_i) = b^{(\ell)}_i \mbox{ for all } i\in [s] \mbox{ and } \ell\in [m+1]] = 1/q^{s(m+1)}.$$ 
Now consider $S^t$, observing that it counts ordered collections of $t$ (potentially overlapping or identical) copies of $H$. Note that the graph $H_t$ spanned by any collection of $t$ copies of $H$ will satisfy $e(H_t)\geq \frac{m}{m+1}|H_t|$, since each connected component of $H_t$ is either a tree with at least $m$ edges or contains a cycle. Therefore, for $d$ sufficiently large in terms of $m$ and $t$, 
$$\E[S^{t}]\leq \sum_{s=m+1}^{(m+1)t}\frac{s^{(m+1)t}q^{m s}}{q^{\frac{m}{m+1}s(m+1)}}< (2mt)^{4mt}.$$
Since $W$ and $D$ are defined over $\mathbb{F}_q$ and their complexities are bounded by a function of $d$ and $m$, Lemma~\ref{lem:27} implies  that there exists a constant $K=K(d, m)$ such that the random variable $S=|W(\F_q)\setminus D(\F_q)|$ satisfies either $S < K$ or $S > q/2$. Thus, by Markov's inequality,
$$\Prob[S \geq K]=\Prob[S>q/2]=\Prob[S^t>(q/2)^t]\leq \frac{\E[S^t]}{(q/2)^t} < (2mt)^{4mt} (q/2)^{-t}.
$$  
Hence, as in the case of cycles, writing $k = (m+1)! K$, we obtain that with high probability the colouring has fewer than $k$ copies of $H$ for each colour type and so it satisfies (Q1) with probability $1-o(1)$.

To see that it also satisfies (Q$2'$), we have to show that with high probability, for each fixed $i \in \F_q^m$ and $a=(a_\ell)_{\ell\in[m+1]}\in \F_q^{m+1}$, the systems of equations 
$$g_\ell(x,i)=a_\ell \mbox{ for } \ell\in[m+1]$$ 
and  
$$g_\ell(i,y)=a_\ell \mbox{ for } \ell\in[m+1]$$ 
have at most $k'/2$ solutions for some $k'=k'(d)$. For this, we again resort to the random algebraic method.

Fix $i \in \F_q^m$. For each $\ell\in [m+1]$, define $g'_\ell(x):=g_\ell(x,i)$. It is not hard to see that $g'_\ell(x)$ is a uniformly random $m$-variable polynomial of degree $d$. Note also that the set 
\begin{equation}\label{eq:sprime}
\{x\in \F_q^m: g_\ell'(x)=a_\ell \text{ for all }\ell\in[m+1]\}	
\end{equation}
is of the form $W'(\F_q)$ for a variety $W'$ of complexity bounded by a function of $d$ and $m$. Moreover, by Lemma~\ref{lem:23} and the independence of the $g'_\ell$, for any $s \leq d+1$ distinct points $x_1,\dots,x_s\in \F_q^m$ and any choice of $a_j^{(\ell)}\in \F_q$ for each $j\in[s]$ and $\ell\in[m+1]$, we have
$$\Prob[g'_\ell(x_j)=a_j^{(\ell)} \mbox{ for all } j\in[s] \mbox{ and }\ell\in[m+1]]=1/q^{s(m+1)}.
$$ 
Consider the random variable $S'=|W'(\F_q)|$ and observe that $S'^t$ counts ordered collections of $t$ (not necessarily distinct) solutions to~\eqref{eq:sprime}. Therefore, for $d$ sufficiently large in terms of $t$,
$$\E[S'^t]\leq\sum_{s=1}^{t}\frac{t^sq^{sm}}{q^{s(m+1)}}=\sum_{s=1}^{t}\left(\frac{t}{q}\right)^s<1. 
$$

By Lemma~\ref{lem:27}, there exists a constant $K'=K'(d, m)$ such that $S'$ satisfies either $S' < K'$ or $S'>q/2$. Thus, by Markov's inequality, provided $t$ is sufficiently large in terms of $m$,
$$\Prob[S' \geq K'] = \Prob[S' > q/2]=\Prob[S'^t>(q/2)^t]\leq \frac{\E[S'^t]}{(q/2)^t} < \frac{2^t}{q^{t}}<\frac{1}{q^{3m+1}}.
$$

In other words, the system of equations
$g_\ell(x,i)=a_\ell$ with $\ell\in[m+1]$ has with high probability at most $K'$ solutions. Together with the analogous result for the system of equations $g_\ell(i,y)=a_\ell$ with $\ell\in[m+1]$, a union bound over all $i\in \F_q^m$ and $a=(a_\ell)_{\ell\in[m+1]}\in \F_q^{m+1}$ establishes (Q$2'$) with high probability for $k'=2K'$.  By another simple application of the union bound, with positive probability the colouring satisfies both properties (Q1) and (Q$2'$), as claimed. 
\end{proof}	

We finish this section by completing the proof of Theorem~\ref{thm:f2}. Indeed, Theorem~\ref{thm:LLL} implies that $f_2(n,C_6)=O(n^{5/3})$, while the next lemma says that $f_2(n,C_6)=\Omega(n^{4/3})$. Together, these results establish Theorem~\ref{thm:f2}(iv) with $H = C_6$.

\begin{theorem}\label{thm:C6}
$f_2(n,C_6)=\Omega(n^{4/3})$.
\end{theorem}

\begin{proof}
Suppose that $C=\gamma n^{4/3}$, where $\gamma$ is a small constant. Suppose also that $n$ is taken sufficiently large and we have a $C$-colouring of $G=K_n$. For simplicity in our notation, we will also assume that $n$ is even. Take an arbitrary equipartition of $V(G)$ into parts $X$ and $Y$. Let $U:=\binom{X}{2}$, $W:=\binom{Y}{2}$ and let $F$ be the following auxiliary graph, similar to the graph used in the proof of Theorem~\ref{thm:paths}. The vertex set is $V(F):=U\cup W$ and $F=F[U,W]$ is bipartite with $uw\in E(F)$ for $u=\{x_1, x_2\}$, $w=\{y_1, y_2\}$ if and only if at least one of $\{x_1y_1,x_2y_2\}$, $\{x_1y_2,x_2y_1\}$ is a monochromatic matching in $G$. Therefore, 
$$|F|=2\binom{n/2}{2}=(1+o(1))\frac{n^2}{4} \geq \frac{n^2}{5}$$
and $e(F)\geq E_2/2$,
where $E_2$ is the number of monochromatic matchings of size $2$ between $X$ and $Y$, as every such matching gives rise to an edge of $F$ and each edge is counted at most twice. By the convexity of binomial coefficients, we obtain 
\begin{equation}\label{eq:eF}
e(F)\geq\frac{E_2}{2}\geq \frac{C}{2}\binom{\frac{n^2}{5C}}{2}>\frac{n^4}{200C}=\frac{n^{8/3}}{200\gamma}>\frac{|F|^{4/3}}{200\gamma}.
\end{equation}

The \emph{theta graph} $\theta_{3,\ell}$ is the bipartite graph composed of $\ell$ internally-disjoint paths of length $3$ sharing the same pair of endpoints. For example, $\theta_{3,2}=C_6$. A theorem of Faudree and Simonovits~\cite{FS} (see also the recent paper~\cite{BuTa}) states that, for every $\ell \geq 2$, $\textrm{ex}(n,\theta_{3,\ell})=O_\ell(n^{4/3})$. Hence, by this and~\eqref{eq:eF}, putting $\ell=60$, there exists  $\gamma_0$ such that if $\gamma<\gamma_0$, then $F$ contains a copy of $\theta_{3,60}$, noting that, since $F$ is bipartite, its endpoints cannot be in the same part.

Suppose now that $u\in U$ and $w\in W$ are the endpoints of our copy of $\theta_{3,60}$, the neighbours of $u$ and $w$ are $w_1\dots,w_\ell$ and $u_1,\dots,u_\ell$, respectively, and $u w_i u_i w$ is a path for each $i = 1, 2, \dots, \ell$. We claim that, after relabelling, there are three indices $1, 2, 3$ such that the eight vertex pairs in $G$ encoded by the vertices $u,w,u_1,u_2,u_3,w_1,w_2,w_3$ are disjoint. Since $u,u_1,u_2,u_3\in U$ and $w,w_1,w_2,w_3\in W$, we only need to make sure that the vertex pairs are disjoint in each part. 

To this end, note, by the same hypergraph colouring argument used in the proof of Theorem~\ref{thm:paths}, that any set of $t$ vertices $w_{i_1},\dots, w_{i_t} \in N_F(u)$ contains a subset of size at least $t/3$ in which all vertices correspond to disjoint pairs in $G$. The same holds for any set of $t$ vertices in  $N_F(w)$. 

Thus, at least $15$ of the vertices $w_1,\dots,w_\ell$ correspond to pairs in $Y$ that are disjoint from each other and from $w$ (since there are at least $20$ of the $w_i$ whose corresponding pairs are disjoint from each other and $w$ overlaps with at most two of these). Relabelling, we can assume that these vertices are $w_1,\dots,w_{15}$. Applying the same argument to $u_1,\dots, u_{15}$ and relabelling if necessary, we obtain three vertices $u_1,u_2,u_3$  corresponding to pairs in $X$ which are disjoint from each other and from $u$.

Let $u:=\{x^1,x^2\}$, $w:=\{y^1,y^2\}$, $u_i:=\{x_i^1,x_i^2\}$ and $w_i:=\{y_i^1,y_i^2\}$ for $i=1,2,3$. By the definition of $F$, without loss of generality we may assume that we have monochromatic pairs $\{x^1y_i^1,x^2y_i^2\}$ and $\{y_i^1x_i^1,y_i^2x_i^2\}$ for $i=1,2,3$. For each of the three edges $u_iw\in E(F)$, we have two options: either the matching $\{x_i^1y^1,x_i^2y^2\}$ is monochromatic or the matching $\{x_i^1y^2,x_i^2y^1\}$ is. By the pigeonhole principle, there will be two indices, say $i=1,2$, where the same situation occurs. In the case where $\{x_i^1y^1,x_i^2y^2\}$ is monochromatic for $i=1,2$, we have a $2$-repeat of $C_6$ with copies $x^1y_1^1x_1^1y^1x_2^1y_2^1x^1$ and $x^2y_1^2x_1^2y^2x_2^2y_2^2x^2$. If instead, $\{x_i^1y^2,x_i^2y^1\}$ is monochromatic for $i=1,2$, our $2$-repeat of $C_6$ consists of the copies $x^1y_1^1x_1^1y^2x_2^1y_2^1x^1$ and $x^2y_1^2x_1^2y^1x_2^2y_2^2x^2$.
\end{proof}	

\section{Open Problems}\label{sec:outlook}

\subsection*{Trees}
Proposition~\ref{prop:infinite} and Theorem~\ref{thm:evencycle} together imply that for each $m$-edge tree $T$, there exists a constant $k_0 = k_0(T)$ such that, for all $k\geq k_0$, 
$$f_k(n,T)=\Theta(n^{\frac{m+1}{m}}).$$
It would be interesting to prove a tight bound for $k_0$ -- we suspect that the answer might be $m+1$ for any $m$-edge tree. When $T$ is the star with $2$ edges, our Theorem~\ref{thm:3rep} verifies this suspicion. When $T$ is the star with $m$ edges, this problem seems closely related to the Zarankiewicz problem for $K_{m, m+1}$, suggesting that it may already be very difficult for $m \geq 3$. An easier problem might be to give an upper bound for $f_3(n, T)$ when $T$ is an $m$-edge tree which matches the lower bound $\Omega(n^{3/2}/\sqrt{m})$ of Theorem~\ref{thm:paths} up to an absolute constant factor.

\subsection*{Bipartite graphs containing a cycle}
Corollary~\ref{cor:LLLbipartite} tells us that if $e(H)\geq 2|H|-2$, then $f_2(n, H) = O(n)$. How sharp is this bound? Recalling that $\theta_{3,\ell}$ is the bipartite graph consisting of $\ell$ internally-disjoint paths of length $3$ sharing the same pair of endpoints, an extension of Theorem~\ref{thm:C6} implies that, for any $\ell \geq 2$,
\[f_2(n, \theta_{3,\ell}) = \Omega(n^{4/3}).\]
Since $e(\theta_{3, \ell}) = \frac{3}{2} |\theta_{3,\ell}| - 3$, we see that Corollary~\ref{cor:LLLbipartite} cannot be improved to say that $e(H)\geq \frac{3}{2}|H|-3$ implies that $f_2(n, H) = O(n)$. An interesting test case for deciding whether the latter bound can be pushed closer to $2 |H|$ might be to study $f_2(n, K'_t)$, where $K'_t$ is the 1-subdivision of the complete graph $K_t$. The extremal properties of these graphs have received close attention in the recent literature~\cite{CoJaLe, CoLe, Ja} and some of the ideas developed in these papers might also prove useful in our context.

\subsection*{Even cycles}	
There are an abundance of open problems regarding even cycles, with the simplest being whether $f_2(n,C_4)=\Theta(n)$. We have only shown that $f_2(n, C_4) = O(n^{3/2})$, as a particular instance of Theorem~\ref{thm:LLL}. It would also be interesting to determine the correct exponent for $C_6$. That is, determine the exponent $\alpha$, if it exists, such that $f_2(n, C_6) = \Theta(n^\alpha)$. Our results only show that $4/3 \leq \alpha \leq 5/3$. Finally, for longer cycles, we suspect that the exponent approaches $2$. That is, is it true that for every $\varepsilon>0$, there exists $m_0=m_0(\varepsilon)$ such that, for all $m\geq m_0$, $f_2(n,C_{2m})=\Omega(n^{2-\varepsilon})$? This problem seems to bear some relation to the even cycle case of Conjecture~3.1 from~\cite{GrJaNa}. 

\subsection*{Non-bipartite graphs} 
When $H$ is non-bipartite, we know, by the remark after Theorem~\ref{thm:oddcycle}, that $f_2(n, H) = n$ for $n$ odd. On the other hand, when $n$ is even, we have only shown that $f_2(n, H) \leq n + 1$. However, in many cases, this bound can be improved to $f_2(n, H) = n - 1$. To see this, we start with the colouring of $K_{n-1}$ with $n-1$ colours used in the proof of Theorem~\ref{thm:oddcycle} and then consider the unique extension of this colouring to $K_n$. The resulting colouring does not contain repeats of `most' non-bipartite $H$, such as graphs containing two disjoint odd cycles, so that $f_2(n, H) = n - 1$ for these graphs. Similarly, this colouring allows us to conclude that $f_3(n, H) = n -1$ for $n$ even and any non-bipartite $H$. However, it remains to determine $f_2(n, H)$ exactly in the natural case where $n$ is even and $H$ is an odd cycle. Curiously, when $H$ is a triangle, each $1$-factorization of $K_n$ has $\binom{n-1}{3}$ triples of possible colours, which is less than $\binom{n}{3}$, the total number of triangles. Hence, by the pigeonhole principle, there will be two colour-isomorphic triangles, but they may not be vertex disjoint.

\subsection*{Repeats of other patterns}
Our results all yield repeated rainbow copies, where the edges in each copy all receive a different colour, but one could also ask for repeats of other patterns. For instance, what is the smallest number of colours in a proper edge-colouring which does not contain a $2$-repeat of the properly $2$-coloured path of length $3$? From the construction in Theorem~\ref{thm:3rep}, it is not hard to deduce an upper bound of $O(n^{3/2})$, whereas a simple counting argument yields a lower bound of $\Omega(n^{4/3})$. We believe the latter to be tight.  
 
\subsection*{A problem in Ramsey theory} 
Our original motivation for studying repeats came from a problem in generalised Ramsey theory raised by Krueger~\cite[Problem 1.2]{Kr}. His question asks for the minimum number of colours in an edge-colouring of $K_n$ such that every copy of $P_t$, the path with $t$ vertices, where $t$ is odd, contains at least $\frac{t+1}{2}$ colours. In our context, the most closely related problem is to study lower bounds for the appearance, in proper edge-colourings, of the pattern consisting of two repeats of the path $P_{(t+1)/2}$ sharing an endpoint. At present, the best lower bound for Krueger's problem is $\widetilde{\Omega}(n^{3/2})$ for all $t \geq 9$. It would be interesting to determine if the exponent tends to $2$ as $t$ tends to infinity both for this problem and ours.

\subsection*{Hypergraphs}
Finally, we note that there are several ways to generalise our problem to the setting of $r$-uniform hypergraphs, depending on how we define a proper colouring. Indeed, for any $1\leq t<r$, one could ask that any two hyperedges sharing at least $t$ vertices receive distinct colours, resulting in a family of problems. It may also be interesting to study proper colourings of Steiner Triple Systems, looking for repeats of fixed linear $3$-graphs.

\vspace{4mm}

\section*{Acknowledgements}

We are extremely grateful to Sean English and Bob Krueger for spotting an error in an earlier version of this paper and suggesting a fix.

\end{document}